\documentclass[final]{siamart251216}
\usepackage{graphicx}
\usepackage{amsfonts,amsmath} 
\usepackage{amssymb,mathrsfs,mathdots,latexsym} 

\usepackage[square,numbers,sort]{natbib} 
\usepackage{multirow}

\usepackage{array}
\usepackage{booktabs} 
\usepackage{enumitem} 
\usepackage{xcolor} 
\usepackage{microtype} 
\usepackage{algorithm} 
\usepackage{algorithmic} 
\usepackage{multicol} 
\usepackage[utf8]{inputenc}
\usepackage[T2A]{fontenc}
\usepackage[russian,english]{babel}

\definecolor{brilliantrose}{rgb}{1.0, 0.33, 0.64}
\definecolor{myviolet}{rgb}{0.21, 0.0, 0.85}
\definecolor{amethyst}{rgb}{0.6, 0.4, 0.8}
\definecolor{carrotorange}{rgb}{0.93, 0.57, 0.13}
\definecolor{cutepink}{rgb}{1.0, 0.2, 0.6}

\counterwithout{algorithm}{section}

\newsiamremark{remark}{Remark}
\newsiamremark{example}{Example}
\newsiamremark{claim}{Claim}

\usepackage{marvosym}

\def\Cminus{\mathbb C_-}

\def\diag{\mathrm{diag}}

\title{Fast stability tests for Hermitian matrix polynomials}

\author{Vanni Noferini\thanks{Department of Mathematics and Systems Analysis, Aalto University, P.O. Box 11100, FI-00076, Aalto, Finland (\email{vanni.noferini@aalto.fi}). Vanni Noferini was supported by a Research Council of Finland grant (decision number 370932).}
\and Xuzhou Zhan\thanks{Corresponding author. Department of Mathematics, Beijing Normal University at Zhuhai, Zhuhai 519087, PR China (\email{xzzhan@bnu.edu.cn}). Xuzhou Zhan  was supported by the Scientific Research Fund from the National Natural Science Foundation of China (12401489), the Beijing Natural Science Foundation (1244044) and Beijing Normal University at Zhuhai (111032119).}}

\begin{document}

\maketitle

\begin{abstract}
Assessing the asymptotic stability of linear self-adjoint homogeneous systems of differential-algebraic equations requires testing the Hurwitz stability of the associated Hermitian matrix polynomial $P(\lambda)$. Tests for known necessary and sufficient conditions rely on linearizations and eigensolvers, solving matrix equations and testing matrix inequalities, or generalized B\'ezoutians, and scale with either $O(d^2 n^3)$ or $O(d^3n^3)$ complexity, where $d$ and $n$ are the degree and size of $P(\lambda)$, respectively. We establish several novel sufficient conditions for stability, based on the numerical range of $P(\lambda)$. Based on the new results, we propose algorithms with $O(d n^3)$ asymptotic complexity. Our methods rely on very efficient core numerical linear algebra routines, such as the Cholesky decomposition of $n \times n$ matrices or the computation of the largest eigenvalue of $n \times n$ definite pencils. Therefore, a significant computational advantage  can be expected in favor of the proposed approach even for moderate values of $d$ or $n$, and we verify this with numerical experiments.
\end{abstract}



\section{Introduction}\label{sec:intro}

Consider the linear self-adjoint homogeneous system of differential-algebraic equations
\begin{equation}\label{System}
A_{d} y^{(d)}(t)+A_{d-1} y^{(d-1)}(t)+\cdots+A_{0} y(t)=0,\quad A_k=A_k^*\in \mathbb C^{n\times n}, \quad A_d \neq 0,
\end{equation}
where $y(t)$ denotes the output vector, assumed to be a complex vector-valued function of the real variable $t$ and $d$ times continuously differentiable, and $y^{(k)}(t)$ denotes its $k$-th derivative with respect to $t$. These systems, including those where the leading coefficient $A_d$ may be not invertible, naturally arise in engineering and physics, and inherently encode their algebraic constraints \cite{KM}. Among their applications are vibrating systems, electric networks, and acoustics \cite{GLRMP, KM, TM}.
In this context, one says that \eqref{System} is \emph{asymptotically stable} if every solution $y(t)$ (that is, independently of the initial conditions at $t=0$, as long as they are consistent) satisfies $\displaystyle \lim_{t\rightarrow +\infty} y(t)=0$, and that \eqref{System} is \emph{stable} if every solution $y(t)$ is uniformly bounded for all $t \geq 0$. 

Both of these properties can be analyzed via the characteristic matrix polynomial of \eqref{System}, that is, the Hermitian matrix polynomial
\begin{equation}\label{Poly}
P(\lambda)=A_{d} \lambda^{d}+A_{d-1} \lambda^{d-1}+\cdots+A_{0},\quad  A_k=A_k^*\in \mathbb C^{n\times n}, \quad A_d \neq 0.
\end{equation}
In particular, the (asymptotic) stability of \eqref{System} can be studied by locating the finite eigenvalues of \eqref{Poly} (see Section \ref{sec:prelim} for the definition of finite eigenvalues of a matrix polynomial, as well as for other definitions and notations used in this introduction). In particular, the matrix polynomial \eqref{Poly} is called \emph{Hurwitz stable} if it is regular, i.e., $\det P(\lambda) \not\equiv 0$, and its finite eigenvalues all lie in the open left half complex plane; it is called \emph{quasi-stable} if it is regular and its finite eigenvalues all lie in the closed left half plane; and it is called \emph{stable} if it is quasi-stable and all pure imaginary eigenvalues are semisimple. One can prove (Theorem \ref{bigtheorem}) that \eqref{System} is asymptotically stable (resp. stable) if and only if \eqref{Poly} is Hurwitz stable (resp. stable).

This raises the question of how to test the stability of \eqref{Poly}. Reliable algorithms exist for the computation of the eigenvalues of polynomial matrices, most usually via a linearization, i.e., a pencil (linear polynomial) $L(\lambda)=L_1 \lambda + L_0$ that captures all the eigenstructure of $P(\lambda)$ in \eqref{Poly}. Many methods to construct a linearization can be found in the literature \cite{DLPV,MMMM,NNT,NP}; then, one can use an eigensolver for pencils, such as the QZ algorithm \cite{GV,MS} (for regular matrix polynomials) or the staircase algorithm  \cite{Van} (for the case where the user also needs to assess whether the input $P(\lambda)$ is singular, i.e., $\det P(\lambda) \equiv 0$). Alternatively, stability can be tested by the Lyapunov criterion \cite{Stykel}. Namely, if $A_d$ is invertible then \eqref{Poly} is Hurwitz stable if and only if a Hermitian matrix $Y$ exists satisfying $L_1^* Y L_1 \succeq 0$ and $L_0^* Y L_1 +L_1^* Y L_0 \succ 0$;  this criterion can be generalized also to the case of a non-invertible leading coefficient by using projection matrices \cite{Stykel}. A third characterization arises from the theory of generalized B\'{e}zoutians developed by Lerer and Tismenetsky \cite{LT82,NNT}, that is, \eqref{Poly} is Hurwitz stable if and only if a certain Lerer-Tismenetsky B\'{e}zout matrix is positive definite. While all these methods provide necessary and sufficient conditions, and hence are as tight as possible, they all rely on cubic-cost algorithms to be performed on matrices of size $O(dn)$, yielding an $O(d^3n^3)$ overall cost. It is worth remarking that, at least for regular matrix polynomials, the complexity of computing the eigenvalues can be reduced to $O(d^2 n^3)$ via structured versions of QZ \cite{AMRVW}. This results in the paradoxical conclusion that, with the technology available at the time when this article is written, testing exactly whether a regular $P(\lambda)$ is Hurwitz stable by either the Lyapunov or the Lerer-Tismenetsky method is computationally more expensive than directly computing the eigenvalues and verifying whether they lie in the open left half plane or not. To partially mitigate this claim, we note in passing that a possible disadvantage of the na\"{\i}ve approach to just solve the polynomial eigenvalue problem is that, infamously \cite{AMRVW,TM,Van}, this can be numerically way more delicate than attempting Cholesky decompositions to test definiteness of Hermitian matrices. Moreover, the na\"{\i}ve test can be inconclusive if one is unlucky enough to find an ill-conditioned eigenvalue near the imaginary axis.

The paragraph above motivates a quest for sufficient conditions for stability that, albeit not characterizing the set of stable polynomials, are quicker to test numerically. Scattered sufficient conditions for stability exist in the literature \cite{MMW, RB, TM}, but they are often overly restrictive, limiting their applicability. In this paper, we systematically improve on these results by establishing novel sufficient conditions that are either weaker or more broadly applicable (or both) than those previously available. Furthermore, when possible, we  analyze the implication chains between both the new and the old criteria.  We will bear particularly in mind the practical applicability of these tests, aiming for more efficient tests than the (mathematically superior, as they provide ``if and only if'' characterizations rather than just ``if'' sufficient conditions) exact methods mentioned above. In particular,  our new results consist of a linear (with respect to the degree $d$) number of cubic-cost algorithms, each to be performed on  matrices or pencils having size $O(n)$, for an overall $O(dn^3)$ complexity. In some applications, $n$ may be very large \cite{TM}, and therefore even an apparently moderate speedup of a factor $d$ or $d^2$ can in practice translate to a significant saving of computational time.  On the other hand, it is rare to find a practical problem whose degree  is higher than four, and therefore we primarily focus on the range $2 \leq d \leq 4$ and  simply sketch how some of our results can be generalized to higher values of $d$. This focus does not limit the computational value of our work, because, in addition to the lower asymptotic growth in $d$, we also expect the constants to be very low for our method. The reason is that we only rely on extremely efficient core tasks, such as Cholesky decompositions and extractions of the largest eigenvalue of definite pencils; see Section \ref{sec:comasp} for more details.

Our approach is based on the notion of numerical range of a matrix polynomial. Recall from \cite{LR94}  that, for $P(\lambda)$ as in \eqref{Poly}, its numerical range  is defined as
$$
W(P)=\{\mu\in \mathbb C: x^*P(\mu)x=0 \mbox{ for }x\neq 0\in \mathbb C^n\}.
$$
When $P(\lambda)=I \lambda-A$, $W(P)$ reduces to the classical numerical range of a square matrix $A$. If $\mu \in \mathbb{C}$ is a finite eigenvalue of \eqref{Poly}, then there exists a nonzero vector $x$ such that $P(\mu)x=0$. Hence, the numerical range of \eqref{Poly} is a superset of its finite spectrum. Therefore, defining $p_x(\lambda):=x^*P(\lambda)x$, the condition
\begin{equation}\label{px}
\mbox{\rm For all } x\neq 0, \mbox{ \rm the real polynomial }p_x(\lambda)\mbox{ \rm is not identically 0 and is Hurwitz stable}
\end{equation}
is sufficient for the Hurwitz stability of \eqref{Poly}. A necessary condition for \eqref{px} to hold is that the matrices $A_k$ are either all positive semidefinite or all negative semidefinite, for $k=0,\ldots, d$; see, e.g., \cite[Statement A]{KaV} or \cite{P}. Note that these two requirements are essentially equivalent, up to multiplying \eqref{Poly} by $-1$. In this paper, we aim to find sufficient conditions for the stability of \eqref{System} that are based on \eqref{px}. For this reason, with no loss of generality, we are going to systematically assume $A_k \succeq 0$. We note in passing that, while this definiteness requirement is necessary for \eqref{px} to hold, it is not necessary for \eqref{System} to be stable: A simple counterexample is 
\[   P(\lambda) = \begin{bmatrix}
    1 & 0\\
    0 & -1
\end{bmatrix} \lambda + \begin{bmatrix}
    1 & 0 \\
    0 & -1
\end{bmatrix} .  \]

The structure of the paper is as follows. In Section \ref{sec:prelim}, we recall some basic notions and develop elementary results, including a complete analysis for $d=1$. In Section \ref{sec:d3}, we derive sufficient conditions for stability of Hermitian matrix polynomials of degree $2 \leq d \leq 3$. We then extend some of these results to higher degrees in Section \ref{sec:d4}. In Section \ref{sec:comasp}, we analyze how our theory results in efficient algorithms. Numerical experiments are presented in Section \ref{sec:numexp}, and conclusions are drawn in Section \ref{sec:conc}.

\section{Preliminaries}\label{sec:prelim}

We denote by $\Cminus \subsetneq \mathbb{C}$ and ${\rm i}\mathbb R \subsetneq \mathbb{C}$ the open left half-plane and the imaginary axis, respectively. If $A=A^*,B=B^*\in \mathbb C^{n \times n}$ are  Hermitian matrices, we write $A \succ B$ if $A-B$ is positive definite, and $A\succeq B$ if $A-B$ is positive semidefinite. 

Let $P(\lambda)$ be a square matrix polynomial of degree $d$ as in \eqref{Poly}. We say that it is \emph{regular} if $\det P(\lambda)$ is not identically zero, and \emph{singular} otherwise. We also say that it is \emph{Hermitian} if all its coefficients satisfy $A_k=A_k^*$, and that it is \emph{positive} (resp. \emph{nonnegative}) if it is Hermitian and $A_k \succ 0$ (resp. $A_k \succeq 0$) for all $k$.
 Recall \cite{DLPV,DN,GLRMP,NNT,NP,TM} that a complex
number $\mu \in \mathbb{C}$ is a \emph{finite eigenvalue} 
of $P(\lambda)$ if
\begin{equation}\label{def:eig}
{\rm rank} P(\mu)<{\rm sup}_{\lambda\in \mathbb C}{\rm rank} P(\lambda).
\end{equation}
If $P(\lambda)$ is regular, then the algebraic multiplicity of the finite eigenvalue $\mu$ is the multiplicity of $\mu$ as a zero of $\det P(\lambda)$, while its geometric multiplicity is $\dim \ker P(\mu)$. A matrix polynomial $P(\lambda)$ can also have infinite eigenvalues, and this happens precisely when the \emph{reversal} $\lambda^d P\left(\frac{1}{\lambda}\right)$ has a zero eigenvalue; algebraic and geometric multiplicities of $\infty$ are analogously defined as the corresponding quantities for the zero eigenvalue of the reversal. An eigenvalue of $P(\lambda)$ is said to be semisimple if its algebraic multiplicity is equal to its geometric multiplicity. In this paper, we will denote by $\Lambda(P)$ the finite spectrum of $P(\lambda)$, that is, the set of all its finite eigenvalues.

Let us argue in Lemma \ref{lem:nosing} that, if the system \eqref{System} is associated  to a singular matrix polynomial \eqref{Poly}, then it cannot be stable and hence, a fortiori, it cannot be asymptotically stable either.

\begin{lemma}\label{lem:nosing}
    If the system \eqref{System} is (asymptotically) stable, then the associated matrix polynomial \eqref{Poly} is regular.
\end{lemma}
\begin{proof}
    Suppose for a contradiction that $P(\lambda)$ is singular. Then, there exists $x(\lambda) \in \mathbb{C}[\lambda]^n$ such that (i) $P(\lambda) x(\lambda) \equiv 0$ (ii) $x(\lambda) \neq 0$ for every $\lambda \in \mathbb{C}$ \cite[Section 2]{DN}. Now consider the vector differential operator $D:=x\left( \frac{d}{dt}\right)$. Then, for every smooth scalar function $f(t)$, consider $y(t):=D f(t)$. Then, $P(\frac{d}{dt}) y(t) = 0$, that is, $y(t)$ is always a solution of \eqref{System} regardless of the choice of $f(t)$. This means, choosing for example $f(t)=e^t$, that the vector valued function $y(t)=D e^t = e^{t} x(1) \neq 0$ solves \eqref{System}, and hence the system is neither asymptotically stable nor stable because
\[ \lim_{t \to +\infty}  \|y(t) \| = \lim_{t \to +\infty} e^t \|x(1)\| = +\infty.   \]
\end{proof}
Lemma \ref{lem:nosing} justifies assuming that \eqref{Poly} is regular. However, in this paper we will need to consider other matrix polynomials constructed from just certain coefficients of \eqref{Poly}, which may be singular even when \eqref{Poly} is regular. Therefore, it is useful to note that the definition of finite eigenvalue given in \eqref{def:eig} is also valid for singular matrix polynomials. It is also possible to extend the definitions of algebraic and geometric multiplicities to the singular case, via the Smith form of $P(\lambda)$ \cite{DN}, and to still link them to, respectively, the multiplicities of roots of certain polynomials \cite[Theorem 2.3]{DN} and the dimension of certain (quotient) vector subspaces \cite[Section 2.3]{DN}. The details of these more general definitions are not strictly needed in the present manuscript, but an interested reader can consult \cite{DN} or the references therein. 

We now restrict our attention to linear matrix polynomials, often called \emph{pencils}. Two Hermitian pencils $A\lambda-B$ and $C\lambda-D$ are called \emph{congruent} if 
there exists a nonsingular square matrix $X \in GL(n,\mathbb{C})$ such that $$X^*(A\lambda-B)X=C\lambda-D.$$ 

\begin{proposition}\label{ProPencil}
Let $A\lambda-B$ be an $n \times n$ Hermitian pencil such that $A\succeq 0$ and $B\succeq 0$. Then $A\lambda-B$ is congruent to a Hermitian pencil of the form
\begin{equation}\label{StrictEquvalence}
\bigoplus_{i=1}^s (\lambda-\mu_i)\oplus - I_t\oplus 0_{n-s-t},\quad  s,t\geq 0,\ s+t\leq p,
\end{equation}
where $\mu_i\geq 0$ for $i=1,\ldots,s$. 
\end{proposition}

\begin{proof}
Following the proof of \cite[Theorem 10.1]{LR05}, the positive semidefiniteness of $A$ yields that $A\lambda-B$ is congruent to
\begin{equation}\label{StrictEquvalenceI}
\bigoplus_{i=1}^s (\lambda-\mu_i) \oplus \bigoplus_{k=1}^t \delta_k D_{r_k}(\lambda) \oplus 0_{n-s-t},
\end{equation}
where $\delta_k \in \{+1,-1\}$ and 
$$
D_{r_k}(\lambda):=\begin{bmatrix}
  0  & \cdots & \lambda & 1 \\
   \vdots & \iddots & 1 & 0 \\
   \lambda & \iddots & \iddots & \vdots \\
    1& 0 & \cdots & 0
\end{bmatrix}\in \mathbb C[\lambda]^{r_k\times r_k}.
$$
It is easy to prove that the eigenvalues of $D_{r_k}(0)$, counted with multiplicity, are $1$ (repeated $\lceil \frac{r_k}{2} \rceil$ times), and possibly $-1$ (repeated $\lfloor \frac{r_k}{2} \rfloor$ times). If either $r_k>1$ or $\delta_k=1$ (or both), this contradicts the positive semidefiniteness of $B$. Similarly, it must be $\mu_i \geq 0$ for all $i=1,\dots,s$, for otherwise the negative index of inertia of $B$ would be nonzero. Hence, \eqref{StrictEquvalenceI} reduces to \eqref{StrictEquvalence}.
\end{proof}

It is immediate by the definitions that congruent pencils have the same finite eigenvalues. Therefore, a consequence of Proposition \ref{ProPencil} is that any finite eigenvalue of a Hermitian pencil $A\lambda-B$ with $A \succeq 0$ and $B \succeq 0$ must be a nonnegative real number, and coincide with of one of the numbers $\mu_i$ in the statement of Proposition \ref{ProPencil}. Note also that Proposition \ref{ProPencil} implies that such a pencil has at least one finite eigenvalue if and only if its leading coefficient $A \neq 0$. In view of these observations, for a Hermitian pencil $A\lambda-B$ with $0 \neq A \succeq 0$ and $B \succeq 0$,  we henceforth 
denote by $\lambda_{\max}(A,B)$ the largest finite eigenvalue of $A\lambda-B$. Generally, $\lambda_{\max}(A,B) \geq 0$, and it may be zero (for example if $A=1, B=0 \in \mathbb{C}^{1 \times 1}$). We record in Proposition \ref{Pro2} a condition under which the quantity $\lambda_{\max}(A,B)$ is guaranteed to be strictly positive.

\begin{proposition}\label{Pro2}
Under the assumptions of Proposition \ref{ProPencil}, suppose that $\ker A\subseteq \ker B
$. Then $A\lambda-B$ is congruent to the Hermitian pencil
$$
\bigoplus_{i=1}^s (\lambda-\mu_i)\oplus 0_{n-s}.
$$
If, in addition, $B \neq 0$ then, for every $0 \neq x\in \mathbb C^n$ such that
$
x^*Bx>0
$, it holds $
x^*Ax>0
$ and
$$
0<\frac{x^*Bx}{x^*Ax}\leq \lambda_{\max}(A,B).
$$
\end{proposition}

\begin{proof}
In view of Proposition \ref{ProPencil}, there exists a nonsingular matrix $X \in GL(n, \mathbb{C})$ such that $X^*(A\lambda-B)X$ has the form \eqref{StrictEquvalence}. If $t>0$, then $A X e_{s+1}=0 \neq B X e_{s+1}$, contradicting the assumption $\ker A\subseteq \ker B$. Hence, $t=0$ and $A \lambda-B$ is congruent to
\begin{equation*}
    \bigoplus_{i=1}^s (\lambda-\mu_i)\oplus 0_{n-s}.
\end{equation*}

Note now that the existence of a vector $x$ as in the last part of the statement is guaranteed by $B \neq 0$. Define $y:=X^{-1}x$. With no loss of generality, we may assume that the $\mu_i$ are ordered nonincreasingly, i.e., $\mu_1\geq \cdots \geq \mu_s$. Then,
$$
\mu_1 x^* A x - x^* B x = x^*(\mu_1A-B)x=y^*\left(\bigoplus_{i=1}^s (\mu_1-\mu_i)\oplus 0_{n-s}\right)y\geq 0.
$$
Recalling that $x^* B x > 0$, we get both $x^* A x > 0$ and $\displaystyle\lambda_{\max}(A,B) = \mu_1 \geq \frac{x^* B x}{x^* A x}  > 0$.
\end{proof}

In the next sections, we will seek sufficient conditions for a Hermitian matrix polynomial to be Hurwitz stable (or quasi-stable). These conditions will include the assumption of positive semidefinite coefficients, i.e., that the matrix polynomial is nonnegative, for the reasons pointed out in Section \ref{sec:intro}. Before turning to this  task, we state and prove the general Lemma \ref{lem:defreg}, that guarantees regularity under a very simple assumption. Since this assumption is always satisfied by our sufficient conditions for stability, this will allow us to simplify the exposition of our proofs, avoiding to check regularity every single time and just focusing on proving that the eigenvalues all lie in the appropriate region of the complex plane.

\begin{lemma}\label{lem:defreg}
    Let $P(\lambda)$ be given as in \eqref{Poly} with $A_k \succeq 0$ for $k=0,\dots,d$. If $\sum_{k=0}^d A_k \succ 0$, then $P(\lambda)$ is regular. The same conclusion holds if $\sum_{k \in S} A_k \succ 0$, where $S \subseteq \{0,\dots,d\}$ is not empty. In particular, if there exists at least one integer $j \in \{0, \dots, d\}$ such that $A_j \succ 0$, then $P(\lambda)$ is regular.
\end{lemma}

\begin{proof}
    Assume that $P(\lambda)$ is singular. Then there exists a polynomial vector $y(\lambda) \in \mathbb{C}[\lambda]^n$ such that (i) $P(\lambda) y(\lambda) \equiv 0$ (ii) $y(\lambda) \neq 0$ for every $\lambda \in \mathbb{C}$ \cite[Section 2]{DN}. Now let $\mu > 0$ be a positive real number, consider the nonzero vector $x:=y(\mu) \in \mathbb{C}^n$, and define $a_k:=x^* A_k x$ for all $k=0,\dots,d$. We readily get the contradiction
\[ 0 = x^* P(\mu) x = p_{x}(\mu) \geq \min\{1,\mu^d\} \cdot \left(\sum_{k=0}^d a_k \right) > 0.     \]
Finally, if $\sum_{k \in S} A_k \succ 0$ and $x \neq 0$, then
\[ x^* \left( \sum_{k=0}^d A_k \right) x \geq x^* \left( \sum_{k \in S} A_k \right) x > 0.   \]
\end{proof}

In view of Proposition \ref{ProPencil} and Lemma \ref{lem:nosing}, it is immediate to analyze the case of a nonnegative pencil. Proposition \ref{ProPencil} implies that its eigenvalues, if any, must be nonpositive real numbers. We record the more general stability analysis in Corollary \ref{CorPencil} below, and in the next sections we will focus on Hermitian polynomials of higher degree $d>1$.

\begin{corollary}\label{CorPencil}
 Let $P(\lambda)$ be as in \eqref{Poly} with $d=1$, $A_1 \succeq 0, A_0 \succeq 0$. Then, $P(\lambda)$ is stable if and only if it is regular, and it is Hurwitz stable if and only if $A_0 \succ 0$.
\end{corollary}
\begin{proof}
    By definition, Hurwitz stable implies stable implies regular. By Proposition \ref{ProPencil}, for a nonnegative pencil regular implies stable. Both $A_1 \lambda + A_0$ being Hurwitz stable and $A_0 \succ 0$ are equivalent to $\det(A_1 \cdot 0 +A_0) \neq 0$.
\end{proof}

We conclude this section with a formal proof of the equivalences between the different types of stability for the system \eqref{System} and for the polynomial matrix \eqref{Poly}. This is folklore, e.g., \cite[Theorem 4.10]{GLRMP}. However, it is still useful for completeness to include a full statement and proof of Theorem \ref{bigtheorem}, because it is surprisingly hard to find an equally explicit and general statement in the literature; for example, \cite[Theorem 4.10]{GLRMP} only treats asymptotic stability, proves only one implication, and assumes that \eqref{Poly} is monic.

\begin{theorem}\label{bigtheorem}
Let $P(\lambda)$ be a regular matrix polynomial as in \eqref{Poly}. Then, the system \eqref{System} is asymptotically stable (resp. stable) if and only if $P(\lambda)$ is Hurwitz stable (resp. stable).
\end{theorem}

\begin{proof}
    Since $P(\lambda)$ is a regular matrix polynomial, its companion pencil \cite{AMRVW, DLPV, NNT, NP} 
    \[ C_1 \lambda + C_0 : = \begin{bmatrix}
        A_d&0&\dots&\dots&0\\
        0&I&0&\dots&0\\
\vdots&\ddots&\ddots&\ddots&\vdots\\
        0&\dots&0&I&0\\
        0&\dots&0&0&I
    \end{bmatrix}  \lambda + \begin{bmatrix}
        A_{d-1} & A_{d-2} &  \dots & A_1 & A_0\\
        -I &0 & \dots& \dots & 0\\
        0& \ddots & \ddots &&\vdots \\
        \vdots&\ddots&-I&0&0\\
        0&\dots&0&-I&0
    \end{bmatrix} \in \mathbb{C}[\lambda]^{dn \times dn} \]
    is also regular and it shares with $P(\lambda)$ the same eigenvalues and their multiplicities. The solutions $y(t)$ of the differential-algebraic system \eqref{System} correspond to the bottom blocks of the solutions of the linearized system $C_1 z'(t) + C_0 z(t) = 0$, which have the structure $\displaystyle z(t) =   \begin{bmatrix}
y^{(d-1)}(t)\\
\vdots\\
y'(t)\\
y(t)
\end{bmatrix}$. Taking into account that the components of a vector solution of a system of the form \eqref{System} are linear combinations of finitely many scalar functions of the form $f_i(t) = t^{k_i} e^{\mu_i t}$, we conclude that if $\lim_{t \to \infty} y(t)=0$ (resp., if $\|y(t)\|$ is uniformly bounded for all $t \geq 0$) then $\lim_{t \to \infty} z(t)=0$ (resp., then $\|z(t)\|$ is uniformly bounded for all $t \geq 0$). The reverse implications are obvious, and therefore it suffices to prove the statement for $d=1$ and $P(\lambda)=C_1 \lambda + C_0$.

Invoking the Weierstrass canonical form for regular pencils \cite{Gan}, there exist nonsingular matrices $U,V \in GL(n,\mathbb{C})$ such that $U C_1 V = I \oplus N$ and $UC_0V= - (J \oplus I),$ where $J$ is a matrix in Jordan canonical form whose eigenvalues are the finite eigenvalues of $P(\lambda)$, and $N$ is a nilpotent matrix. Let us apply a change of variables and a block partition (coherent with the sizes of $J$ and $N$) by defining $\displaystyle \begin{bmatrix}
    z_1(t)\\
    z_2(t)
\end{bmatrix}:=V^{-1}z(t)$, and let us left-multiply the linearized system by $U$. We then conclude that $z_1'(t) = J z_1(t)$ and $N z_2'(t) = z_2(t)$. Since $N$ is nilpotent, it is easy to verify that the only smooth solution to the latter equation is $z_2(t) \equiv 0$ (and therefore any initial condition to \eqref{System} must be coherent with this constraint). On the other hand, the explicit solution to the former equation is $z_1(t)= e^{t J} z_1(0)$. Therefore, $\lim_{t \to +\infty}  z(t) =0$ if and only if $\lim_{t \to +\infty} e^{t J} = 0$, if and only if all the eigenvalues of $J$ have negative real part \cite{HJ}, if and only if the finite spectrum $\Lambda(P) \subsetneq \Cminus$. Similarly, $\| z(t) \|$ is uniformly bounded if and only if $e^{t J}$ is uniformly bounded, if and only if all the eigenvalues of $J$ have nonpositive real part and those on the imaginary axis are semisimple \cite{HJ}, if and only if the same statement holds for the finite eigenvalues of $P(\lambda)$.
\end{proof}

\section{Sufficient conditions for stability when $2 \leq d\leq 3$}\label{sec:d3}

When the system \eqref{System} is asymptotically stable with $A_k\succeq 0$ for $k=0,\ldots,d$, then necessarily $0\not\in \Lambda(P)$ for $P(\lambda)$ given in \eqref{Poly}. Therefore, in this scenario, it must be $A_0 \succ 0$. Below, we add other assumptions to this necessary one, to obtain sufficient conditions that ensure the asymptotic stability of the system \eqref{System}. To this goal, given $0 \neq x \in \mathbb{C}^n$, it is convenient to define, here and throughout the paper,
\begin{equation}\label{def:ak}
    a_k = x^* A_k x, \qquad k=0,\dots, d.
\end{equation}

Let us start from the cubic case $d=3$.  
 
\begin{theorem}\label{ThmHurwitz3}
The system \eqref{System} with $d=3$ is asymptotically stable if $A_0\succ 0 $, $A_1\succ 0 $, and (at least) one of the following conditions holds:
\begin{enumerate}
    \item
$A_2\succeq A_3\succeq 0$  and $\lambda_{\max}(A_1,A_0)\lambda_{\max}(A_2,A_3)<1$.

    \item 
$A_1\succ A_0 $, $A_2\succeq A_3\succeq 0.$

\item $A_1\succeq A_0 $, $A_2\succ A_3\succeq 0.$

\end{enumerate}
\end{theorem}

\begin{proof}
Let $
P(\lambda)$ be given as in \eqref{Poly} and $a_k$ be as in \eqref{def:ak}. By Lemma \ref{lem:defreg}, $P(\lambda)$ is regular.
\begin{enumerate}
    \item  Suppose first $a_2a_3=0$. Since $A_2\succeq A_3\succeq 0 $, either $a_2=a_3=0$ or $a_2>a_3=0$. 
Hence,
$$p_x(\lambda)=\begin{cases} a_2\lambda^2+a_1 \lambda+a_0, &\mathrm{if} \ a_2>a_3=0; \\
a_1 \lambda+a_0, &\mathrm{if} \ a_2=a_3=0.
\end{cases}$$
Since, in addition, $a_1 >0$ and $a_0 > 0$, we have that $p_x(\lambda)$ is  Hurwitz stable by classical results on scalar polynomials \cite{Gan}.

Now assume that $a_2a_3\neq 0$, or equivalently that $a_2>0$ and $a_3>0$. 
Note that $a_0>0$ and that $\ker A_1 =  \ker A_0 = \{ 0 \}$. Hence, by Proposition \ref{Pro2}, 
\begin{equation}\label{eq1.1}
\frac{a_0}{a_1}\leq \lambda_{\max}(A_1,A_0).
\end{equation}
Moreover, $A_2\succeq A_3\succeq 0 $ implies that
$\ker A_2\subseteq \ker A_3$. Hence,
again by Proposition \ref{Pro2},
\begin{equation}\label{eq1.2}
\frac{a_3}{a_2}\leq \lambda_{\max}(A_2,A_3).
\end{equation}
Combining \eqref{eq1.1}, \eqref{eq1.2}, and the assumptions (1), we conclude that $a_0>0$, $a_2>0$, and $\displaystyle \begin{vmatrix}
        a_2& a_0\\
        a_3& a_1
    \end{vmatrix}=a_1 a_2-a_0 a_3>0$.
In view of the classical stability criterion of Li\'{e}nard and Chipart  \cite[p. 221, Theorem 11]{Gan}, this means that
the real polynomial $p_x(\lambda)$ is Hurwitz stable. Therefore, $P(\lambda)$ is Hurwitz stable and the system \eqref{System}  is asymptotically stable.
\item By assumption, $a_1 > a_0 > 0$ and $a_2 \geq a_3 \geq 0$. If $a_2=0$, then $p_x(\lambda)=a_1 \lambda + a_0$ which is manifestly Hurwitz stable. If instead $a_2 > 0$, then $a_1 a_2 > a_0 a_2 \geq a_0 a_3$ and therefore $p_x(\lambda)=a_3 \lambda^3 + a_2 \lambda^2 + a_1 \lambda + a_0$ is Hurwitz stable by the Li\'{e}nard-Chipart criterion (or by standard results if $a_3=0$).
\item In this case we have $a_1 \geq a_0 > 0$ and $a_2 > a_3 \geq 0$, which leads to the same proof as in the second part of the previous item.
\end{enumerate}

\end{proof}

\begin{remark}
    Switching the roles of $A_3,A_0$ and of $A_2,A_1$ corresponds to taking the reversal of \eqref{Poly}, i.e., considering the matrix polynomial $\lambda^3 P\left(\frac{1}{\lambda}\right) = A_0 \lambda^3 + A_1 \lambda^2 + A_2 \lambda + A_3$. It is known \cite{DLPV, DN, GLRMP, MMMM, NNT, NP} that, except for zero eigenvalues, the finite eigenvalues of the reversal are the reciprocals of the finite eigenvalues of the original polynomial. Since $\Cminus$ is invariant under the mapping $z \mapsto z^{-1}$, we conclude that an analogue of Theorem \ref{ThmHurwitz3} holds after switching the roles of the coefficient matrices in this way, and \emph{provided that the reversal polynomial does not have infinite eigenvalues}, or equivalently that $P(\lambda)$ does not have zero eigenvalues. This subtlety is necessary because $0 \notin \Cminus$, and once again is tantamount to still requiring $A_0 \succ 0$. For example, applying this idea to item 1 in Theorem \ref{ThmHurwitz3}, another valid sufficient condition for asymptotic stability when $d=3$ is given by $A_3 \succ 0, A_2 \succ 0$, $A_1 \succeq A_0 \succ 0$ and $\lambda_{\max}(A_2,A_3) \lambda_{\max}(A_1,A_0)<1$. Although, for the sake of conciseness, we will not repeat an explicit remark, similar comments apply to all the results that we derive in our paper: Up to possibly adding the assumption $A_0 \succ 0$, one can always apply each criterion to the reversal $\lambda^d P\left(\frac{1}{\lambda}\right)$ and obtain an equally valid sufficient condition.
\end{remark}

Our results for cubic polynomials also yield a criterion for quadratic polynomials, that we state in Corollary \ref{Cor32}.

\begin{corollary}\label{Cor32}
The system \eqref{System} with $d=2$ is asymptotically stable 
if $A_0\succ 0 $, $A_1\succ 0$, $A_2\succeq 0 $. 
\end{corollary}
\begin{proof}
    $A_2 \neq 0$, hence $\lambda_{\max}(A_2,0)=0$ and item 1 in Theorem \ref{ThmHurwitz3} applies.
\end{proof}

\begin{remark}
Other stability criteria for quadratic polynomials have appeared in the literature. The  overdamping condition \cite[Equation (3.17)]{TM}, which confines the spectrum to the negative real axis (a much stronger property than Hurwitz stability), requires positive definite coefficients and imposes the nonlinear condition $(x^* A_1 x)^2 > 4(x^* A_2 x)(x^* A_0 x)$ for all $x \neq 0$. In comparison, Corollary \ref{Cor32} ensures asymptotic stability under weaker, and easier to check, requirements. Recently, in \cite[Corollary 18]{RB} Roy and Bora also obtained a sufficient condition for quadratic polynomials, applicable also to non-Hermitian coefficients, and based on the inequality $\|(A_0 + A_1 + A_2)^{-1}\|_p^{-1} > \| A_0 - A_1 + A_2 \|_p + 2 \| A_0 - A_2 \|_p$ for $p \in \{1, 2, \infty\}$. The Roy-Bora condition can be tested in $O(n^3)$ complexity, similarly to ours. Corollary \ref{Cor32} and \cite[Corollary 18]{RB} do not imply each other. Consider for example the family of stable scalar quadratic polynomials such that $A_0=A_1=1$ and $2 \leq A_2$, whose roots have real part  $ -(2 A_2)^{-1} \in [-\frac14,0[$. For these polynomials, Corollary \ref{Cor32} applies but \cite[Corollary 18]{RB} fails. Conversely, it is also possible to construct stable Hermitian quadratic polynomials with indefinite coefficients (and hence beyond the applicability of Corollary \ref{Cor32})  that satisfy the inequality of \cite[Corollary 18]{RB}.
\end{remark}

We now focus on the case when the system \eqref{System} has positive definite coefficients and $d=3$.

\begin{lemma}\label{LemHurwitz3}
The system \eqref{System} with $d=3$ is asymptotically stable 
if $A_k\succ 0 $ for $k=0,1,2,3$ and (at least) one of the following conditions holds:
\begin{enumerate}
    \item There exists $\theta\in [0,1]$ such that
\begin{equation}\label{Lem3.2.1}
\lambda_{\max}(A_1,\theta A_0+(1-\theta)A_3)\lambda_{\max}(A_2,(1-\theta)A_0+\theta A_3)<1.
\end{equation}
\item There exists $\theta\in [0,1]$ such that
\begin{equation}\label{Lem3.2.2}
A_1\succeq \theta A_0+(1-\theta)A_3, \ A_2\succeq (1-\theta)A_0+\theta A_3,
\end{equation}
in which at least one inequality is strict.
\end{enumerate}
\end{lemma}

\begin{proof}
Let $
P(\lambda)$ be given as in \eqref{Poly} and $a_k$ as in \eqref{def:ak}. By Lemma \ref{lem:defreg}, $P(\lambda)$ is regular.
\begin{enumerate}
    \item  By Proposition \ref{Pro2},
\begin{align}
&\frac{\theta a_0+(1-\theta)a_3}{a_1}\leq \lambda_{\max}(A_1,\theta A_0+(1-\theta)A_3), \label{prima} \\
&\frac{(1-\theta)a_0+\theta a_3}{a_2}\leq \lambda_{\max}(A_2,(1-\theta)A_0+\theta A_3). \label{seconda}
\end{align}
Combining \eqref{prima} and  \eqref{seconda} with \eqref{Lem3.2.1}, we conclude that $a_0 a_3 < a_1 a_2$. Since, in addition, $a_0 > 0$ and $a_2 > 0$, it follows by the Li\'{e}nard-Chipart criterion that condition \eqref{px} holds. Hence, $P(\lambda)$ is Hurwitz stable and the system \eqref{System}  is asymptotically stable.
    \item  It follows from \eqref{Lem3.2.2} that 
$$
\lambda_{\max}(A_1,\theta A_0+(1-\theta)A_3)\leq 1,\quad \lambda_{\max}(A_2,(1-\theta)A_0+\theta A_3) \leq 1,
$$
and at least one inequality is strict. Hence, \eqref{Lem3.2.1} holds. The statement then follows from item 1.
\end{enumerate}
\end{proof}

\begin{corollary}\label{Cor36}
The system \eqref{System}  with $d=3$ is asymptotically stable 
if $A_k\succ 0 $ for $k=0,1,2,3$,  and (at least) one of the following conditions holds:
\begin{enumerate}
      \item
$\lambda_{\max}(A_1, A_3)\lambda_{\max}(A_2, A_0)<1.$
      \item 
$\lambda_{\max}(A_1, A_0+A_3)\lambda_{\max}(A_2, A_0+A_3)<4.$
\item 
 $A_1\succeq A_3$,  $A_2\succeq A_0 $, in which at least one inequality is strict.

\item $2 A_1\succeq A_0 +A_3$, $2 A_2\succeq A_0 +A_3$, in which at least one inequality is strict.
\end{enumerate}
\end{corollary}

\begin{proof}
It suffices to specialize Lemma \ref{LemHurwitz3} to $\theta=0$ and $\theta=\frac12$.
\end{proof}

Now we turn to analyzing the stability of the system \eqref{System}
with $d \leq 3$.  

\begin{lemma}\label{LemQuasi}
Let $d=3$ and let $P(\lambda)$ be given as in \eqref{Poly}. Then $P(\lambda)$ is quasi-stable 
if 
$A_1\succeq A_0\succeq 0$,  $A_1+A_2+A_3\succ 0$, and (at least) one of the following statements holds:
\begin{enumerate}
    \item$A_2\succeq A_3\succeq 0$.
    \item $A_2\succeq A_0$, $A_1\succeq A_3\succeq 0$.
\end{enumerate}
\end{lemma}

\begin{proof}
By Lemma \ref{lem:defreg}, $P(\lambda)$ is regular. Define $a_k$ as in \eqref{def:ak}.
\begin{enumerate}
    \item We have $a_1+a_2+a_3 > 0$. Moreover, $a_1 \geq a_0 \geq 0$ and $a_2 \geq a_3 \geq 0$, implying $a_1 a_2 \geq a_0 a_3$. Suppose first $a_2=0$, implying $a_3=0$ and $a_1 > 0$. Then,  $p_x(\lambda)=a_1 \lambda + a_0$, which is also quasi-stable since $a_0 \geq 0$. If instead $a_2 > 0$, write $p_x(\lambda)=p_x^{[e]}(\lambda^2)+\lambda p_x^{[o]}(\lambda^2)$ with 
    $$
p_x^{[e]}(\lambda):=a_2\lambda+a_0,\quad p_x^{[o]}(\lambda):=a_3\lambda+a_1.
$$
If $a_2 a_1 > a_0 a_3$ is a strict inequality, then $-a_3 p_x^{[e]}(\lambda)+a_2 p_x^{[o]}(\lambda) = a_2 a_1 - a_3 a_0 \neq 0$, showing that $p_x^{[e]}$ and $p_x^{[o]}$ are coprime and $p_x(\lambda)$ is quasi stable by \cite[Theorem 4.9]{AGT}. If instead $a_2 a_1 = a_3 a_0$, then there are two subcases. If $a_3 >0$, then $p_x^{[e]}$ is a nonzero constant multiple of $p_x^{[o]}$, and both polynomials have a unique root $-\frac{a_0}{a_2} = - \frac{a_1}{a_3} \leq 0$, implying quasi-stability by \cite[Theorem 4.9]{AGT}. If $a_3=0$, then $0=a_1\geq a_0=0$, and $p_x(\lambda)=a_2 \lambda^2$, which is manifestly quasi-stable.

\item In this case, we have $a_1+a_2+a_3 > 0$, $a_1 \geq a_3 \geq 0$, and $a_2 \geq a_0 \geq 0$, which also imply $a_1 a_2 \geq a_0 a_3$. From this point, we can follow almost the same argument as for item 1, except for checking that $p_x(\lambda)=a_3 \lambda^3 + a_1 \lambda$ is quasi-stable when $a_3,a_1>0$.
\end{enumerate}
\end{proof}

\begin{remark}\label{rem:comparison}
   Item 1 in Lemma \ref{LemQuasi} strictly improves  on the condition given in \cite[Theorem 6.8]{MMW}, which assumes $A_3 \succ 0$, $A_2 \succ 0$, $A_0 \succ 0$, $A_1 \succeq 0$, $A_2 + A_1 \succ 0$, $A_2 \succeq A_3$, and $A_1 \succeq A_0$. Indeed, suppose that the assumptions of \cite[Theorem 6.8]{MMW} hold. Then, $A_1 \succeq A_0 \succ 0$, $A_2 \succeq A_3 \succ 0$, and $A_1 + A_2 + A_3 \succ 0$ all hold, and hence the assumptions of item 1 in Lemma \ref{LemQuasi} hold. The reverse implication does not hold, as can be seen by taking $A_3=A_2=A_1=I$ and $A_0=0$.
\end{remark}

\begin{corollary}\label{CorQuasi}
 Let $d=3$ and let $P(\lambda)$ be given as in \eqref{Poly}. $P(\lambda)$ is quasi-stable if $A_0\succeq 0$, $A_1\succ 0$, $A_2\succeq A_3\succeq 0$, and $\lambda_{\max}(A_1,A_0)\lambda_{\max}(A_2,A_3)\leq 1$.
\end{corollary}
\begin{proof}
    The assumptions yield the same inequalities as in the proof of Lemma \ref{LemQuasi}.
\end{proof}

\begin{corollary}\label{CorQuasi2}
Let $d=3$ and let $P(\lambda)$ be given as in \eqref{Poly}. Then $P(\lambda)$ is quasi-stable 
if $A_k\succ 0 $ for $k=0,1,2,3$, and (at least) one of the following statements is satisfied:
\begin{enumerate}
    \item
  $\lambda_{\max}(A_1, A_0)\lambda_{\max}(A_2, A_3)\leq 1.$

      \item 
$\lambda_{\max}(A_1, A_3)\lambda_{\max}(A_2, A_0)\leq 1.$
      \item
$\lambda_{\max}(A_1, A_0+A_3)\lambda_{\max}(A_2, A_0+A_3)\leq 4.$
\item
 $A_1\succeq A_3$,  $A_2\succeq A_0 $.

\item $2 A_1\succeq A_0 +A_3$, $2 A_2\succeq A_0 +A_3$.
\end{enumerate}
\end{corollary}

\begin{proof}
    It is easy to see that each of the 
    five conditions yields $a_2 a_1 \geq a_0 a_3$. The proof then continues as in Lemma \ref{LemQuasi}.
\end{proof}

In Theorem \ref{ThmStability}, we provide some sufficient conditions for the system \eqref{System} to be stable when $d=3$.

\begin{theorem}\label{ThmStability}
Let $d=3$. The system \eqref{System} is stable 
if $A_0\succeq 0 $, $A_1\succ 0 $, and (at least) one of the following statements is satisfied:
  \begin{enumerate}
 
 \item
$0 \neq A_2\succeq A_3\succeq 0$,  $\lambda_{\max}(A_1,A_0)\lambda_{\max}(A_2,A_3)\leq 1$.
\item $A_1\succeq A_0$, $A_2\succeq A_3\succeq 0$.
    \item $A_1\succeq A_0$, $A_2\succeq A_0$, $A_1\succeq A_3\succeq 0$.

\end{enumerate}
\end{theorem}

\begin{proof}
Let $
P(\lambda)$ be given as in \eqref{Poly} and $a_k$ as in \eqref{def:ak}. In view of Lemma \ref{LemQuasi} and its proof, $
P(\lambda)$ and $p_x(\lambda)$ are both quasi-stable, and in view of Lemma \ref{lem:defreg}, $P(\lambda)$ is regular. If $\Lambda(P)\subseteq \mathbb C_-$, then system \eqref{System} is trivially asymptotically stable and hence stable, so it suffices to prove the statement assuming $\Lambda(P)\cap {\rm i}\mathbb R\neq \emptyset$. Suppose that ${\rm i}b\in \Lambda(P)\cap {\rm i}\mathbb R$.  In this case,
$$
{\rm Re}({\rm i}b)=0=\max\{{\rm Re}\ \lambda: \lambda\in W(P)\}, 
$$
which means that ${\rm i}b$ is a boundary point of $W(P)$. 

We now follow in part the proof of \cite[Theorem 1.1]{Wimmer}.
Let $0 \neq v \in \ker P({\rm i}b)$.
By \cite[Theorem 1.1]{MP}, $0$ is a boundary point of the numerical range of the constant matrix $P({\rm i}b)$. Then, it follows from \cite[Theorem 1.6.6]{HJ} that $P({\rm i}b)$ is unitarily similar to $0 \oplus B$, where $B$ is a nonsingular matrix. Then $v ^*P({\rm i}b)=0$. We now claim that
\begin{equation}\label{vP'v}
v^*P'({\rm i}b) v\neq 0.
\end{equation}
Then, there is no possible solution $w \in \mathbb{C}^n$ to the equation $P'({\rm i}b) v+P({\rm i}b) w=0$, implying that i$b$ is a semisimple eigenvalue of $P(\lambda)$ \cite{DN,GLRMP}. Thus, $P(\lambda)$ is stable and so is \eqref{System}.

It remains to prove \eqref{vP'v}. To this goal, specialize the definitions of $a_k$ in \eqref{def:ak} to $x=v$, so that $a_k : = v^* A_k v$. If $b=0$, \eqref{vP'v} is tantamount to $a_1 \neq 0$, which is true because $A_1 \succ 0$. If instead $b \neq 0$, suppose for a contradiction $v^*P'({\rm i}b) v=0$, and observe that $v^*P({\rm i}b) v=0$. Expanding and taking real and imaginary parts of the latter two equations, we obtain the linear system
\begin{equation*}
 \begin{bmatrix}
  1 & -b^2 & 0  & 0\\ 
   0 & 2b & 0 & 0\\
   0 & 0 & b & -b^3\\
  0 & 0 & 1 & -3 b^2\\
\end{bmatrix}\begin{bmatrix}
    a_0\\
    a_2\\
    a_1\\
    a_3
\end{bmatrix} = 0 \Rightarrow a_0 = a_1 = a_2 = a_3 = 0,
\end{equation*}
 which is a contradiction because, by the positive definiteness of $A_1$, $a_1 > 0$. 
\end{proof}

We now state Corollaries \ref{CorStable} and \ref{CorHurwitz}. Their proofs are very similar to that of Theorem \ref{ThmStability}, and for this reason we only sketch them.

\begin{corollary}\label{CorStable}
Let $d=3$. The system \eqref{System} is stable 
if $A_k\succ 0 $ for $k=0,1,2,3$, and (at least) one of the following statements is satisfied:
\begin{enumerate}
    \item
  $\lambda_{\max}(A_1, A_0)\lambda_{\max}(A_2, A_3)\leq 1.$

      \item
$\lambda_{\max}(A_1, A_3)\lambda_{\max}(A_2, A_0)\leq 1.$
      \item
$\lambda_{\max}(A_1, A_0+A_3)\lambda_{\max}(A_2, A_0+A_3)\leq 4.$
\item 
 $A_1\succeq A_3$,  $A_2\succeq A_0 $.

\item $2 A_1\succeq A_0 +A_3$, $2 A_2\succeq A_0 +A_3$.
\end{enumerate}
\end{corollary}

\begin{proof}
  By Corollary \ref{CorQuasi2}, $P(\lambda)$ is quasi-stable, so it suffices to show that every pure imaginary eigenvalue of $P(\lambda)$ (if any) is semisimple. To this goal, suppose there exists $b \in \mathbb{R}$ such that ${\rm i}b \in \Lambda(P)$, and let $0 \neq v \in \ker P({\rm i}b)$. Setting $a_k := v^* A_k v > 0$, and following the proof of Theorem \ref{ThmStability}, assume for a contradiction that $v^* P'({\rm i}b) v =0$, implying $a_2 b = 0$. Then, either $a_2=0$, contradicting $A_2 \succ 0$, or $b=0$, contradicting $A_0 \succ 0$.
\end{proof}

\begin{corollary}\label{CorHurwitz}
Let $d=2$. The system \eqref{System} is stable 
if $A_k\succeq 0 $ for $k=0,1,2$, $A_0+A_1\succ 0 $.
\end{corollary}

\begin{proof}
    By Lemma \ref{lem:defreg}, $P(\lambda)$ is regular. Let $a_k$ be as in \eqref{def:ak}. It is an easy exercise to check that $a_k \geq 0$ imply that $p_x(\lambda)=a_2 \lambda^2 + a_1 \lambda +a_0$ is quasi-stable.

    Suppose now that there exists $b \in \mathbb{R}$ such that ${\rm i}b \in \Lambda(P)$ and let $0 \neq v \in \ker P({\rm i}b)$. We claim that $v^* P'({\rm i}b) v \neq 0$ and conclude as in Theorem \ref{ThmStability}. To prove the claim, specialize the definition of $a_k$ in \eqref{def:ak} to $x=v$, i.e., $a_k := v^* A_k v$. Note that $\Im (v^* P({\rm i}b) v) = a_1 b = 0$. If $b=0$, then $a_0=0$. Therefore $a_1 = a_0 + a_1 > 0$, and hence $v^* P'(0) v = a_1 \neq 0$. If $b \neq 0$, then $a_1 = 0$ and therefore $a_0 = a_0 + a_1 > 0$. Moreover, $\Re (v^* P({\rm i}b) v) = a_0 - b^2 a_2 = 0$, implying $a_2 = a_0 b^{-2} > 0$ and in turn that $v^* P'({\rm i}b) v = 2 {\rm i} a_2 b \neq 0$.
\end{proof}

\section{Sufficient conditions for Hurwitz stability when $d \geq 4$}\label{sec:d4}

Some of the techniques that we used in Section \ref{sec:d3} extend to higher degrees, leading to similar results. In the present section, we explain how and why. Some of the arguments needed follow quite closely those that we used for $d \leq 3$. For this reason, and for the sake of conciseness, we occasionally only sketch certain steps. We also limit ourselves to studying Hurwitz stability of \eqref{Poly}. Obtaining conditions for the stability of \eqref{System}, except for those strong enough to guarantee asymptotic stability, is a difficult problem for high degrees because, given $b \in \mathbb{R}$ and $0\neq v \in \ker P({\rm i}b)$, the higher the degree the more involved the algebraic condition $v^* P'({\rm i} b) v = 0$.

\begin{theorem}\label{ThmHurwitz4}
The system \eqref{System}  with $d=4$ is asymptotically stable 
if $A_k \succ 0$ for all $k=0,\dots,3$, $A_3 \succeq A_4 \succeq 0$, and (at least) one of the following statements is satisfied:
\begin{enumerate}
    \item 
$
  \lambda_{\max}(A_1, A_0)\lambda_{\max}(A_2, A_3)+\lambda_{\max}(A_2,A_1)\lambda_{\max}(A_3,A_4)<1.
$
   \item
$
  \lambda_{\max}(A_1, A_3)\lambda_{\max}(A_2, A_0)+\lambda_{\max}(A_2,A_1)\lambda_{\max}(A_3,A_4)<1.
$
\end{enumerate}
\end{theorem}

\begin{proof}
By Lemma \ref{lem:defreg}, $P(\lambda)$ is regular. Let $a_k$ be as in \eqref{def:ak}. By assumption, $a_0, a_1, a_2, a_3 > 0$ and $a_3 \geq a_4 \geq 0$.

By repeated applications of Proposition \ref{Pro2} (observe that $\ker A_1 = \ker A_2 = \{0\}$ and that $A_3 \succeq A_4 \succeq 0$ implies $\ker A_3 \subseteq \ker A_4$) we get 
\[ \frac{a_0}{a_1} \leq \lambda_{\max}(A_1,A_0), \ \frac{a_3}{a_2} \leq \lambda_{\max}(A_2,A_3),  \ \frac{a_1}{a_2} \leq \lambda_{\max}(A_2,A_1)   \ \mathrm{and} \ \frac{a_4}{a_3} \leq \lambda_{\max}(A_3,A_4).\] Multiplying these bounds,
\[ \frac{a_0 a_3}{a_1 a_2} \leq \lambda_{\max}(A_1, A_0)\lambda_{\max}(A_2, A_3), \quad \frac{a_1 a_4}{a_2 a_3} \leq \lambda_{\max}(A_2,A_1)\lambda_{\max}(A_3,A_4). \]
Then, it is clear that 
either of the two enumerated conditions 
ensures that $\displaystyle \frac{a_0 a_3}{a_1 a_2} + \frac{a_1 a_4}{a_2 a_3} < 1$. Clearing denominators, this is in turn equivalent to $a_1 a_2 a_3 - a_0 a_3^2 - a_1^2 a_4 > 0$. Under the condition $a_k > 0$, this is equivalent to the Li\'{e}nard-Chipart criterion \cite[p. 221]{Gan} for quartic polynomials.
\end{proof}

\begin{theorem}\label{ThmHurwitzVanni}
Let $z_0 \approx 2.1479$ be the unique positive solution to the polynomial equation $z^3 - z^2 - 2z - 1=0$. The system \eqref{System} is asymptotically stable if $A_k \succ 0$ for all $k=0,\dots,d$,
 and (at least) one of the following statements is satisfied:
\begin{enumerate}
    \item $\displaystyle \lambda_{\max}(A_k, A_{k-1})\lambda_{\max}(A_{k+1}, A_{k+2}) \leq \frac{1}{z_0}$ for all $k=1,\ldots,d-2$.
    \item 
$\displaystyle
 \sum_{k=1}^{d-2} \lambda_{\max}(A_k, A_{k-1})\lambda_{\max}(A_{k+1}, A_{k+2})<1.
$
\end{enumerate}
\end{theorem}
\begin{proof}
By Lemma \ref{lem:defreg}, $P(\lambda)$ is regular. Let $a_k$ be as in \eqref{def:ak}. By assumption, $a_k > 0$ for $k=0,\dots,d$. Moreover, by repeated applications of Proposition \ref{Pro2}, we have $\displaystyle \frac{a_{k-1}}{a_k} \leq \lambda_{\max}(A_k, A_{k-1})$ and $\displaystyle \frac{a_{k+2}}{a_{k+1}} \leq \lambda_{\max}(A_{k+1}, A_{k+2})$ for all $k=1,\dots,d-2$. (Note that all the coefficients $A_k$ are invertible, and thus the null spaces inclusion, needed to apply Proposition \ref{Pro2}, holds trivially.) 
\begin{enumerate}
\item For every internal index $k=1, \dots, d-2$, it holds
\[ \frac{a_{k-1} a_{k+2}}{a_k a_{k+1}} \leq \lambda_{\max}(A_k, A_{k-1})\lambda_{\max}(A_{k+1}, A_{k+2}) \leq \frac{1}{z_0}. \]
This implies $a_k a_{k+1} \geq z_0 a_{k-1} a_{k+2}$ for all $k=1,\ldots,d-2$. By the Katkova-Vishnyakova criterion,  proved in \cite[Theorem 1]{KaV}, this in turn guarantees that the real polynomial $p_x(\lambda) = \sum_{k=0}^d a_k \lambda^k$ is Hurwitz stable. Thus, $P(\lambda)$ is Hurwitz stable and the system \eqref{System} is asymptotically stable. 
\item By Lemma \ref{lem:defreg}, $P(\lambda)$ is regular. Let $a_k$ be as in \eqref{def:ak}. By assumption, $a_k> 0$ for $k=0,\ldots,d$. 
By repeated applications of Proposition \ref{Pro2}, we have  
$\displaystyle 0<\frac{a_{k-1}}{a_k} \leq \lambda_{\max}(A_k, A_{k-1})$ and $\displaystyle 0<\frac{a_{k+2}}{a_{k+1}} \leq \lambda_{\max}(A_{k+1}, A_{k+2})$ for all $k=1,\ldots,d-2$. Hence,
\[ \sum_{k=1}^{d-2} \frac{a_{k-1} a_{k+2}}{a_k a_{k+1}}    \leq \sum_{k=1}^{d-2} \lambda_{\max}(A_k, A_{k-1})\lambda_{\max}(A_{k+1}, A_{k+2})<1.  \]
Therefore, by Kleptsyn's criterion \cite[Теорема]{Kl}, $p_x(\lambda)$ is Hurwitz stable; see also \cite[Section 1]{AGT} for a source in English.  Thus, $P(\lambda)$ is Hurwitz stable and the system \eqref{System} is asymptotically stable.
\end{enumerate}
\end{proof}

\section{Computational aspects}\label{sec:comasp}

In the present section, we discuss how to convert the theoretical analysis of Sections \ref{sec:d3} and \ref{sec:d4} into numerical algorithms implementable in floating point arithmetic, having in mind potentially large-scale instances of the system \eqref{System}. To motivate the design of efficient algorithms, we note that the value of the size $n$ of the matrices $A_k$ in \eqref{System} may be very large in practical applications \cite{KM,TM}. The sufficient conditions that we derived in the previous sections rely on two core numerical linear algebraic tasks: (1) Assessing the positive definiteness of Hermitian matrices; (2) Computing the largest finite eigenvalue of positive, and hence definite\footnote{Recall that a Hermitian pencil $\lambda A + B$ is definite if $(x^* A x)^2 + (x^* B x)^2 > 0$ for all $x \neq 0.$}, pencils.

We observe in passing that, although some of our theoretical conditions require only positive semidefiniteness, this property is numerically undecidable for a generic Hermitian matrix. There are of course exceptions, of which an obvious one is a matrix of the form $A = B \oplus 0$, for which verifying $B \succ 0$ provides a test for $A \succeq 0$. However, excluding such special cases, the standard practical approach is to test for strict positive definiteness. Thus, to avoid complicating the computational discussion with overly subtle special cases, both in the present section and in Section \ref{sec:numexp} we henceforth focus on positive definite coefficients.

Both tasks 1 and 2 above can be implemented using highly efficient dense numerical linear algebra algorithms \cite{ASNA, GV}:
\begin{enumerate}
    \item The positive definiteness of $M = M^* \in \mathbb{C}^{n \times n}$ is verified by attempting a Cholesky decomposition $M = LL^*$ for a lower triangular $L \in \mathbb{C}^{n \times n}$, requiring only $\approx \frac{1}{3}n^3$ flops. The test is successful if the algorithm terminates. Of course, when executed in floating point arithmetic, the algorithm can be affected by approximation errors. Thus, in a sense, one can state a ``numerical version" of the sufficient conditions for stability. To this goal, let $\lambda$ be the smallest eigenvalue of the matrix $H=\diag(M)^{-1/2} M \diag(M)^{-1/2}$ where $\diag(M)$ is a diagonal matrix coinciding with $M$ on the diagonal elements. Then, a result due to Demmel \cite[Theorem 10.7]{ASNA}
    states that the Cholesky algorithm is guaranteed to succeed if $\lambda > c_n$ and guaranteed to fail if $\lambda \leq - c_n$. Here, $c_n \approx n^2$\textbf{u} as long as $n$\textbf{u} $\ll 1$, where $n$ is the size of $M$ and \textbf{u} is the unit roundoff. To get a rough idea of the concrete sensitivity for large-scale systems, note that double precision floating point arithmetic corresponds to \textbf{u}$=2^{-53}$, and thus $c_n \approx 1.1 \cdot 10^{-10}$ for $n=10^3$ and $c_n \approx 1.1 \cdot 10^{-4}$ for $n=10^6$. The Cholesky factor itself can provide an a posteriori estimate of the test's reliability, e.g., by checking that $\min_i L_{ii}/\max_i L_{ii}$ is not too small compared to \textbf{u}. 
    \item The largest eigenvalue of the definite pencil $A\lambda - B$ can be computed in two steps. First, we compute the Cholesky factorization $A = LL^*$ and then we extract the largest eigenvalue of the Hermitian matrix $L^{-1} B L^{-*}$. The latter step can be performed very efficiently, for example, by means of the Lanczos algorithm \cite{GV} or the power method.
\end{enumerate}

\subsection{Comparison against previous methods}

Evaluating the sufficient conditions proposed in Theorem \ref{ThmHurwitzVanni} requires an at most $d+1$ Cholesky decompositions of $n \times n$ Hermitian matrices and the computation of the largest eigenvalue of $2d-4$ definite pencils; this results in an overall $O(dn^3)$ complexity. The other novel sufficient conditions, applicable to small degrees, also scale linearly in $d$ and cubically in $n$. This provides a clear computational advantage with respect to existing methods.

Let us briefly review why. The companion pencil approach scales as $O(d^2n^3)$ if an ad hoc structured algorithm is employed \cite{AMRVW}. In practice, for moderate degrees we expect the speedup to far exceed the asymptotic factor $d$, due to favorable constants. For instance, the companion approach computes all $dn$ eigenvalues of a non-Hermitian generalized eigenvalue problem, whereas the proposed method isolates a single dominant eigenvalue of a definite pencil, which can be done far more efficiently. Alternatively, if we compare with the approach of Lerer and Tismenetsky \cite[Section 2]{LT82}, their first step requires solving the equation $Q^*(-\lambda)Q(\lambda)=P^*(-\lambda)P(\lambda)$ for a matrix polynomial $Q(\lambda)$ coprime with $P(\lambda)$ \cite{GLRMP}. It is unclear, in general, how to solve this equation other than by first computing the eigenvalues and eigenvectors of $P(\lambda)$. Even assuming that a more direct technique is found, a second step then requires the Cholesky factorization of a $dn \times dn$ matrix, for an overall $O(d^3 n^3)$ complexity. A third approach is via the Lyapunov criterion described in Section \ref{sec:intro}. Directly tackling  the linear matrix inequalities is possible \cite{LMI} but computationally inefficient. Assuming for simplicity that $L_1$ (and hence $A_d$) is invertible, via the change of variables $X=L_1^* Y L_1, A=L_1^{-1}L_0$ we can equivalently solve the Lyapunov equation $A^* X + X A = I$ for the Hermitian unknown $X$ and test whether $X \succeq 0$. (It can be proved \cite{Stykel} that a positive semidefinite solution exists for right hand side $I$ if and only if it exists for every positive definite right hand side $Q \succ 0$.)  This method requires solving a matrix equation of size $dn \times dn$, plus a Cholesky decomposition, and it therefore also scales as $O(d^3n^3)$.

\subsection{Comparison between different new methods}

In many cases, we expect a negligible performance gap between the methods solely based on Cholesky decompositions and those also requiring the computation of one eigenvalue of definite pencils. This is because, if the power method on the matrix $L^{-1} B L^{-*}$ is employed, then each step requires solving two triangular systems and computing one matrix-vector multiplication, for an overall $O(n^2)$ cost per iteration. If the spectral gap between the dominant and the next eigenvalue is not too small, this extra cost is dominated by the $O(n^3)$ operations to compute the Cholesky factors. If instead the spectral gap is small (this happens for example, with high probability for large $n$, for the randomly generated matrices in Section \ref{sec:numexp}, whose eigenvalues are Marchenko-Pastur distributed), a computational advantage can be expected for the conditions that do not require eigenvalue computations. However, a good reason to nevertheless also consider the pencil-based criteria is that they are strictly weaker than those solely based on positive definiteness, and therefore applicable to a strictly larger set of systems \eqref{System}. We illustrate this fact with Proposition \ref{prop:implication}.

\begin{proposition}\label{prop:implication}
\begin{enumerate}
    \item In Theorem \ref{ThmHurwitz3}, items 3 and (provided $A_2 \neq 0$) 2 imply item 1. The converse implications do not hold.
    \item In Corollary \ref{Cor36}, item 3 implies item 1, and item 4 implies item 2. The converse implications do not hold.
\end{enumerate}

\end{proposition}

\begin{proof}
\begin{enumerate}
    \item Suppose item 2 holds, i.e., $A_1 \succ A_0$ and $A_2 \succeq A_3 \succeq 0$. Let $\mu:=\lambda_{\max}(A_1, A_0)$, then there exists a vector $x \in \mathbb{C}^n$ such that $x^*A_0x=1$ and $0=\mu x^* A_1 x - x^*A_0x > \mu -1$, having used $A_1 \succ A_0$. Hence, $\mu < 1$.  Similarly, $0 \neq A_2 \succeq A_3 \succeq 0$ implies $\lambda_{\max}(A_2, A_3) \leq 1$. (Note that $A_2 \neq 0$ is necessary to make sure that $\lambda_{\max}(A_2,A_3)$ is defined.) We conclude that item 1 is true. The proof that item 3 implies item 1 is similar and omitted. To disprove that item 1 implies any other, let $n=1$, $A_0 = A_2= 10$, $A_1 = 5$, and $A_3 = 1$.
\item We prove that item 4 implies item 2 but not vice versa, omitting the similar argument that shows the implication chain between items 3 and 1. Suppose item 4 holds. Then, $2 A_1 \succeq A_0 + A_3$ implies $\lambda_{\max}(A_1, A_0+A_3) \leq 2$. Analogously, $2 A_2 \succeq A_0 + A_3$ yields $\lambda_{\max}(A_2, A_0+A_3) \leq 2$. Since at least one inequality is strict, item 2 holds. To disprove the reverse implication, let $n=1$ and take $A_0 = 4, A_1 = 10$, $A_2 = A_3 = 1$.
\end{enumerate}
\end{proof}

\section{Numerical experiments}\label{sec:numexp}

In this section, we set up a number of numerical experiments in Monte Carlo style, to confirm the expected computational advantages of our approach. In most experiments, we compare one or more algorithms based on some of our new results with (i) the direct computation of all the eigenvalues of the companion pencil $C_1 \lambda + C_0$ (ii) the Lyapunov criterion, again tested on the companion linearization. For (i), we have used MATLAB's native {\tt polyeig} function; since {\tt polyeig} does not fully exploit the companion structure, the reader should take into account as a rule of thumb that the method in \cite{AMRVW}, for which we did not have an implementation available, is expected to be about $d$ times faster than {\tt polyeig}. (The experimental data will clarify that this factor $d$ is rather irrelevant to the conclusion that our methods are definitely much faster.) For (ii), we rely on the function {\tt lyap} available in MATLAB's Control System Toolbox. The method described in \cite{LT82}, which is very implicit making it unclear how to implement it in general, was not included in the comparison for practical reasons. All experiments have been performed in MATLAB R2019b on a MacBook Air equipped with MacOS 26.3. Except when explicitly specified otherwise,  we have generated the Hermitian coefficients $A_k$ as\footnote{We only present experiments on real symmetric matrices, both for conciseness and because this case seems more frequent in applications.  We observed similar results for Hermitian coefficients.}
\[ \texttt{ A\{k\} = randn(n); A\{k\}=A\{k\}'*A\{k\}/n + c*eye(n);} \] Here, $c=2 c_n$ and $c_n$ is the constant of \cite[Theorem 10.7]{ASNA}, discussed in Section \ref{sec:comasp}. Note that $\mathbb{E}[\diag(A_k-c I)]=I$. Hence, via e.g. the Chernoff inequality, the Cholesky decomposition on $A_k$ succeeds in double precision with overwhelming probability.

\subsection{$d=2$}

We compare a test based on Corollary \ref{Cor32}, i.e., three Cholesky decompositions, against (i) and (ii), as well as against a test based on \cite[Corollary 18]{RB} (for $p=\infty$). We report the results in Table \ref{tab:d2}; as some methods became too slow in practice for the statistical nature of the experiment, we excluded them for large input size. In practice, our procedure was as follows: We flagged when an algorithm reached an average execution time beyond $12$ seconds (this threshold is predictive, given the theoretical analysis of an $O(n^3)$ complexity, of a total execution time beyond $2$ hours in the next step). Flagged algorithms were then excluded from the comparison for larger values of $n$, and we just report N/A in the corresponding entries of the table.
\begin{table}[htbp]
\centering
\begin{tabular}{r|r|r|r|r}
\toprule
$n$ & Cor. \ref{Cor32} &  {RB} & \texttt{polyeig} & \texttt{lyap} \\
\midrule
200 & 0.0011 & 0.0028 & 0.1294 &  0.1437\\
500  & 0.0041 & 0.0110 & 2.3081 & 1.5723 \\
1000 & 0.0185 & 0.0542 & 25.950 & 7.8352 \\
2000 & 0.1143 & 0.3154 & N/A &  51.748 \\
5000 & 2.3196 & 5.2035 & N/A & N/A \\
10000 & 24.033 & 56.967 & N/A & N/A \\
\bottomrule
\end{tabular}
\caption{Average execution time (in seconds) over $100$ random trials for stability tests ($d=2$).}
\label{tab:d2}
\end{table}

\subsection{$d=3$}

Using an experimental set up similar to that for $d=2$, we compare against (i) and (ii) three tests based on, resp., items 1 and 2 in Theorem \ref{ThmHurwitz3} and item 2 in Corollary \ref{Cor36}. Note that a criterion based on item 1 in Theorem \ref{ThmHurwitz3} includes testing whether $A_2 \succeq A_3$. This property is not guaranteed by how we generate the matrices $A_k$, and the probability that it holds tends to $0$ as $n \rightarrow \infty$. Hence, terminating the algorithm based on Theorem \ref{ThmHurwitz3} after that particular test fails would generate an artificial and unfair advantage for that approach; for this reason, we always force the algorithm to also check the condition of item 1, based on the largest eigenvalues of four definite pencils. The results are in Table \ref{tab:d3}.

\begin{table}[htbp]
\centering
\label{tab:d3}
\begin{tabular}{r|rrr|r|r}
\toprule
$n$ & Thm. \ref{ThmHurwitz3}(1) & Thm. \ref{ThmHurwitz3}(2) & Cor. \ref{Cor36}(2) & \texttt{polyeig} & \texttt{lyap} \\
\midrule
200 & 0.0051 & 0.0014 & 0.0031 & 0.5741 & 0.5201 \\
500 & 0.0177 & 0.0032 & 0.0152 & 10.738 & 3.6263 \\
1000 & 0.0759	 & 0.0105 & 0.0725
 & 119.98 & 25.167 \\
2000 & 0.4730	 & 0.0812 & 0.4683 & N/A & N/A \\
5000 & 5.5545 & 1.4402 & 5.5384 & N/A & N/A \\
10000 & 81.940 & 13.700 & 84.813 & N/A & N/A\\
\bottomrule
\end{tabular}
\caption{Average execution time (in seconds) over $100$ random trials for stability tests ($d=3$).}
\end{table}

We next run a second experiment on cubic matrix polynomials, aimed at probing the relative likelihood of success between our methods based (also) on eigenvalues of definite pencils and those purely based on Cholesky decompositions. As a test case, we compare items 1 and 2 in Theorem \ref{ThmHurwitz3}, and we set the matrix size to the small value $n=10$, allowing us to execute a large number of trials in very little time. In order to create ill-conditioned coefficients, and hence a more demanding test for both sufficient conditions, for this experiment we generate the matrices $A_k$ by the command

\[ \texttt{[Q,R] = qr(randn(n));Q=Q*diag(sign(diag(R)));}\]
\[\texttt{A\{k\} = Q *diag(10.\textasciicircum(sigma * randn(n, 1))) * Q';} \]
        This ensures that the eigenvectors of $A_k$ follow the Haar (uniform) distribution on the sphere. The eigenvalues of $A_k$ are lognormally distributed, and $\sigma \in [0.2,0.4]$ is the standard deviation of their base-$10$ logarithm. As a second step, to further increase the probability that the constructed polynomials are Hurwitz stable, we multiply the central coefficients $A_1,A_2$ by a scaling factor $\gamma > 1$. Still, not all the matrix polynomials thus created are Hurwitz stable, which is undesirable in the context of this experiment because we wish to empirically estimate the probability \emph{conditioned to the input being Hurwitz stable} that items 1 and 2 in Theorem \ref{ThmHurwitz3}  are successful; otherwise, at least for certain values of $\gamma$ and $\sigma$, the unconditioned success probabilities may well be too low to provide statistically significant information about the performance gap. For this reason, as a next step, we use {\tt polyeig} as an oracle to pre-select $1000$ Hurwitz stable inputs, setting a threshold of $-10^{-10}$ for the real parts of the eigenvalues. In the final stage of the numerical experiment, we count how many Hurwitz stable inputs are correctly identified by either test. Table \ref{tab:mc3} reports the results for $15$ particular choices of $(\gamma,\sigma)$. The two parameters heavily influence the probability that the tests for stability given in Theorem \ref{ThmHurwitz3} are successful. More in detail, the smaller the value of $\gamma$ or the higher the value of $\sigma$, the more probable that the variance of the log-normal distribution causes a violation of the conditions $A_1 \succ A_0, A_2 \succ A_3$ as well as a growth of the product $\lambda_{\max}(A_1, A_0) \lambda_{\max}(A_2, A_3)$. Hence, we expect the absolute success rate of both the sufficient conditions under scrutiny to substantially decrease under these conditions. Table \ref{tab:mc3} confirms this prediction. However, precisely in this ``stress test'' regime, the pencil-based criterion (item 1 in Theorem \ref{ThmHurwitz3}) remarkably achieves a high relative success rate with respect to the more efficient, but more rigid, Cholesky-based criterion (item 2 in Theorem \ref{ThmHurwitz3}). As $\gamma$ increases or $\sigma$ decreases, the absolute detection rates of both tests predictably increase, and the ratio of success probabilities also converges towards $1$ from above. (Recall from Proposition \ref{prop:implication} that the ratios must be bounded below by $1$ for mathematical reasons.)

\begin{table}[htbp]
\centering
\setlength{\tabcolsep}{4.5pt}
\label{tab:mc3}
\begin{tabular}{@{}c|ccc|ccc|ccc@{}}
\toprule
 & \multicolumn{3}{c|}{$\sigma = 0.2$} & \multicolumn{3}{c|}{$\sigma = 0.3$} & \multicolumn{3}{c}{$\sigma = 0.4$} \\
\cmidrule(lr){2-4} \cmidrule(lr){5-7} \cmidrule(lr){8-10}
$\gamma$ & Penc & Chol & Ratio & Penc & Chol & Ratio & Penc & Chol & Ratio \\
\midrule
2 & 79 & 28 & 2.82 & 10 & 2 & 5.00 & 3 & 0 & $\infty$ \\
3 & 638 & 361 & 1.77 & 93 & 36 & 2.58 & 27 & 5 & 5.40 \\
5 & 993 & 921 & 1.08 & 571 & 325 & 1.76 & 164 & 63 & 2.60 \\
7 & 1000 & 992 & 1.01 & 860 & 641 & 1.34 & 399 & 189 & 2.11 \\
10 & 1000 & 999 & 1.00 & 987 & 892 & 1.11 & 708 & 435 & 1.63 \\
\bottomrule
\end{tabular}
\caption{Number of successes, out of $1000$ Hurwitz stable inputs, of sufficient stability conditions for $n=10$, $d=3$, across varying $\sigma$ and $\gamma$. Penc and Chol denote, resp. items 1 and 2 in Theorem \ref{ThmHurwitz3}.}
\end{table}

\subsection{$d=4$}

For the quartic case, still following the same basic ideas as for $d=2,3$, we set up an experiment to compare the running time of a test based on item 1 in Theorem \ref{ThmHurwitz4} with (i) and (ii). Since the criterion based on item 1 in Theorem \ref{ThmHurwitz4} includes testing whether $A_3 \succeq A_4$, we always force the algorithm to also check the other conditions, including the one based on the largest eigenvalues of four definite pencils. Table \ref{tab:d4} summarizes the outcome.

\begin{table}[htbp]
\centering
\label{tab:d4}
\begin{tabular}{r|r|r|r}
\toprule
$n$ & Thm. \ref{ThmHurwitz4}(1) & \texttt{polyeig} & \texttt{lyap} \\
\midrule
100 & 0.0011 & 0.1910 & 0.3178 \\
200 & 0.0021 & 1.5175 & 1.0034 \\
500 & 0.0079 & 29.986 & 7.6943 \\
1000 & 0.0351 & N/A & 55.181 \\
2000 & 0.2161 & N/A & N/A \\
5000 & 3.9394 & N/A & N/A \\
\bottomrule
\end{tabular}
\caption{Average execution time (in seconds) over $100$ random trials for stability tests ($d=4$).}
\end{table}

\subsection{Large $d$}

The previous experiments already provide ample evidence that the performance gap between the newly proposed sufficient conditions and the traditional methods is substantial, and grows as $d$ does. For this reason, we do not attempt a comparison with (i) and (ii) for very large values of $d$, because we can expect our competitors to be excruciatingly slow in this regime. Instead, we simply measure the average execution time for an algorithm based on Theorem \ref{ThmHurwitzVanni}, for varying $d$ and $n \in \{ 100, 200 \}$. Note that the execution time is identical for item $1$ and item $2$ in Theorem \ref{ThmHurwitzVanni}, because they share the same core computations and only differ in the final inequality assessment, whose computational cost is of course negligible. Figure \ref{fig:larged} displays the results and is coherent with the predicted $O(dn^3)$ complexity.

\begin{figure}[htbp]
\centering
\includegraphics[width=0.78\textwidth]{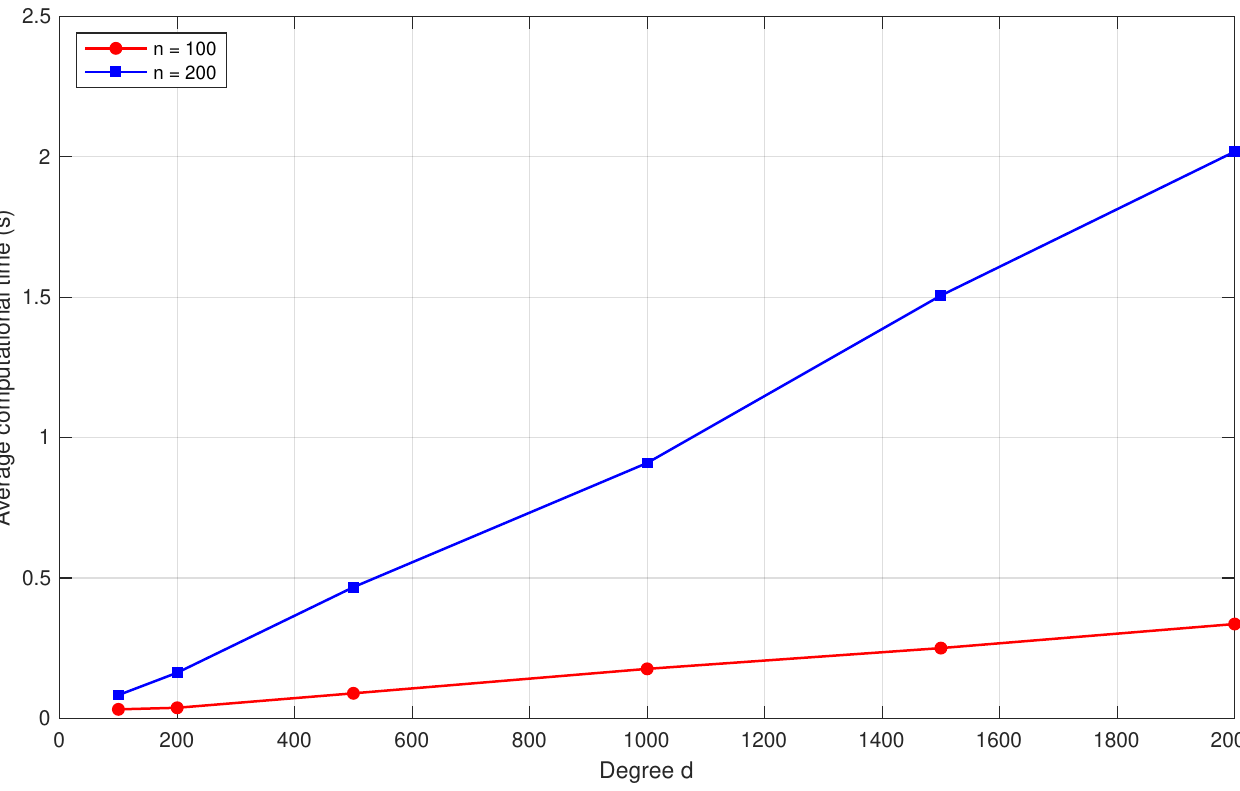}
\caption{Average execution time (in seconds) over $100$ trials to evaluate the sufficient conditions of Theorem \ref{ThmHurwitzVanni} as a function of the polynomial degree $d$. Matrix size $n = 100$ (red circles) and $n = 200$ (blue squares), degree $d = 100, 200, 500, 1000, 1500, 2000$.}
\label{fig:larged}
\end{figure}

\section{Conclusions}\label{sec:conc} 

Assessing the (Hurwitz) stability of a Hermitian matrix polynomial has applications to engineering and physics. Analytically, the problem was solved long ago \cite{LMI,LT82,Stykel}, but the resulting numerical algorithms can be uneconomical. In this paper we have systematically reviewed and improved the state of the art on efficiently computable sufficient conditions.

\section*{Acknowledgements}

We thank Matvei Zhukov who helped us translate \cite{Kl}.

\end{document}